\documentclass{article}

\usepackage[T1]{fontenc}     
\usepackage{amsthm}  
\usepackage{amsmath}
\usepackage{amsfonts}
\usepackage{tikz}
\usepackage[capitalize, noabbrev]{cleveref}
\usepackage{tikz-cd}
\usepackage{amssymb}
\usepackage{filecontents}
\usepackage{graphicx}
\usepackage{subcaption}
\usepackage{float}

\newtheorem{theo}{Theorem}[section]
\newtheorem*{thaupt}{Haupt's theorem}

\newtheorem{cor}[theo]{Corollary}
\newtheorem{prop}[theo]{Proposition}
\newtheorem{lemma}[theo]{Lemma}
\theoremstyle{remark}{}

\theoremstyle{definition}

\newcommand{\Per}{\mathrm{Per}}
\newcommand{\C}{\mathbb C}
\newcommand{\Z}{\mathbb Z}
\newcommand{\R}{\mathbb R}

\newcommand{\Hom}{\mathrm{Hom}}

\newcommand{\vol}[1]{\mathrm{Vol}(#1)}

\newcommand{\SL}[1]{\mathrm{SL}_2 (#1)}
\newcommand{\Mod }{\mathrm{Mod}}

\author{Thomas Le Fils}
\date{}

\begin{document}
\title{Periods of abelian differentials with prescribed singularities}

\maketitle

\begin{abstract}
We give a necessary and sufficient condition for a representation of the fundamental group of a closed surface of genus at least $2$ to $\C$ to be  the holonomy of a translation surface with a prescribed list of conical singularities. Equivalently, we determine the period maps of abelian differentials with prescribed list of multiplicities of zeros. Our main result was also obtained, independently, by Bainbridge, Johnson, Judge and Park.
\end{abstract}

\section{Introduction}

Let $S$ be a closed oriented surface of genus $g\geqslant 2$. Denote by $\Gamma$ a fundamental group of $S$. 

If $X$ is a complex structure on $S$, an \textit{abelian differential} on $X$ is a holomorphic 1-form on $X$. Let us denote by $\Omega$ the space of $(X, \alpha)$ where $X$ is a complex structure on $S$ and $\alpha$ an abelian differential on $X$ and by $\Omega'$ the subset of $\Omega$ where $\alpha$ is not the zero form. The \textit{period} of $\alpha\in \Omega$ is the homomorphism $\chi\in \Hom(\Gamma, \C)$ defined by $\chi(\gamma)= \int_\gamma \alpha$. We thus define a map
$$\Per : \Omega\to \Hom(\Gamma, \C).$$

We define the volume of $\chi\in \Hom(\Gamma, \C)$ to be $\sum \Im(\overline{\chi(a_i)}\chi(b_i))$, where $(a_1, b_1, \ldots, a_g, b_g)$ is a standard symplectic basis of $\Gamma$.
In \cite{Haupt}, Haupt computed the image of $\Per$ by $\Omega'$.

\begin{thaupt}
The image of $\Omega'$ by $\Per$ is the set of $\chi\in \Hom(\Gamma, \C)$ such that
\begin{enumerate}
\item $\vol \chi > 0$
\item  if $\Lambda = \chi(\Gamma)$ is a lattice in $\C$, then $\vol \chi \geqslant 2\mathrm{Area}(\C / \Lambda)$.
\end{enumerate}
\end{thaupt}
Let us recall that the sum of the degrees of zeros of an abelian differential is $2g-2$. Therefore, the space $\Omega'$ is naturally stratified by the subsets $\mathcal H(n_1, \ldots, n_k)$ consisting of $\alpha\in \Omega'$ with zeros of multiplicity $n_1, \ldots n_k$, where $\sum n_i = 2g-2$. As a refinement of Haupt's theorem, we characterise the period maps of abelian differentials in a given stratum, thus answering a question raised by Calsamiglia, Deroin and Francaviglia in \cite{deroin}.

\begin{theo}\label{resultat}
Let $1 \leqslant n_1\leqslant n_2, \ldots \leqslant n_k$ be such that $\sum n_i = 2g-2$. The image of $\mathcal H(n_1, \ldots, n_k)$ by $\mathrm{Per}$ is the set of $\chi \in \Hom(\Gamma, \C)$ satisfying: \begin{enumerate}
\item $\vol \chi > 0$
\item if $\Lambda = \chi(\Gamma)$ is a lattice in $\C$, then $\vol \chi \geqslant(n_k+1) \mathrm{Area}(\C /\Lambda)$.

\end{enumerate}
\end{theo}

These obstructions where already observed in \cite{deroin}. The fact that every $\chi\in \Hom(\Gamma, \C)$ of positive volume and non-discrete image is in the image of $\Per$ by every stratum is a consequence of their work. Our proof will nevertheless include this case.

In order to explain the origins of these obstructions on the volume, it is convenient to introduce a more geometric point of view on this problem.
A \textit{translation surface}, or \textit{flat surface} structure on $S$ is the datum of an atlas of chart in the euclidean plane $\mathbb E^2$, with transition maps that are translations. We allow the charts to have conical points, that is  to have the local form $z\mapsto z^k$. More precisely, a translation surface structure is a branched $(G,X)$ structure on $S$ where $G = \C$ is the group of translations of $X = \mathbb E^2$. The local charts globalize into a \textit{developing map}  $f : \tilde{S}\to \C$, where $\tilde{S}$ is the universal cover of $S$, that satisfies $f(\gamma\cdot z) = \chi(\gamma) + f(z)$ for some $\chi\in \Hom(\Gamma, \C)$ called the \textit{holonomy} of  the translation surface (see \cite[Chapter 3]{Thurston} for more information on $(G, X)$ structures). If $\alpha\in \Omega'$, then one can define local charts on $S$ by $\int_{z_0}^z \alpha$. In this manner, we get a one-to-one correspondence between translation surfaces and abelian differentials, see \cite[Section 3.3]{Zorich}. The period map of an abelian differential corresponds to the holonomy of the associated flat surface, and the stratum $\mathcal H(n_1, \ldots, n_k)$ corresponds to the flat surfaces with $k$ conical points of angles $2\pi(n_1+1), \ldots, 2\pi(n_k+1)$. Hence the question adressed in this article is to characterise the $\chi\in \Hom(\Gamma, \C)$ which are holonomies of flat surfaces with a prescribed list of conical angles.

Suppose that $\chi\in \Hom(\Gamma, \C)$ is the period map of $\alpha\in \Omega'$. The Riemann bilinear relations (see \cite[Section 14]{Narasimhan}) give:
$$\frac{i}{2}\int_S \alpha\wedge\overline\alpha = \frac{i}{2} \sum_{i=1}^g \int_{a_i} \alpha \int_{b_i}\overline\alpha - \int_{a_i} \overline\alpha \int_{b_i} \alpha = \sum_{i=1}^g \Im(\overline{\chi(a_i)}\chi(b_i)).$$\label{computation}

Let $z = x + iy$ be a local coordinate on $U\subset S$, and write $\alpha = f(z)dz$.  We have $\int_U |f|^2dxdy = \frac{i}{2}\int_U \alpha\wedge\overline{\alpha} > 0$. A first obstruction to being the holonomy of a flat structure is thus to have positive volume.

Suppose now moreover that $\Lambda = \chi(\Gamma)$ is a lattice in $\C$. Then the developing map of the translation surface induces a holomorphic map $S\to \C/\Lambda$. If $\alpha$ has a zero of degree $n$, then the degree $d$ of this branched cover is at least $n+1$. The volume of $\chi$ is $d\cdot\mathrm{Area}(\C/\Lambda)$. Therefore a second obstruction to being the holonomy of $\alpha\in \mathcal H(n_1, \ldots, n_k)$ is to have $\vol \chi \geqslant (n_k+1) \mathrm{Area}(\C / \Lambda)$ when $\Lambda = \chi(\Gamma)$ is a lattice in $\C$.

\cref{resultat} implies the existence of branched cover of elliptic curves with prescribed induced map on homology.
\begin{cor}\label{corcor}

Let $\Lambda$ be a lattice in $\C$ and $1\leqslant n_1\leqslant \ldots \leqslant n_k$. 
There exists a branched cover $S\to \C/\Lambda$ of degree $d$ with branching points of order $n_1, \ldots, n_k$ if and only if $\sum n_i = 2g-2$ and  $d\geqslant n_k+1$.
Moreover, if $\chi \in \Hom(\Gamma, \Lambda)$, then $\chi$ is induced by one of these covering maps if and only if \begin{enumerate}
\item $\vol \chi = d \cdot\mathrm{Area}(\C/\Lambda)$
\item $\vol \chi \geqslant (n_k+1)\mathrm{Area}(\C/\Lambda')$, where $\Lambda' = \chi(\Lambda)$.
\end{enumerate}
\end{cor}

\begin{proof}
Let us consider a covering map $S\to \C/\Lambda$ of degree $d$ with branching points of order $n_1, \ldots, n_k$. Its degree is at least $n_k+1$. The pullback $\alpha$ of $dz$ by this branched cover is in $\mathcal{H}(n_1, \ldots, n_k)$ and thus $\sum n_i = 2g-2$. The period map of $\alpha$ is the map induced on homology by this covering, and has volume $d\cdot\mathrm{Area}(\C/\Lambda)$ by the computation of $\int \alpha\wedge\overline{\alpha}$ above. Observe that $\Lambda' = \chi(\Gamma)$ is a lattice since it is a subgroup of $\Lambda$ and $\vol\chi > 0$. The developing map of the translation surface associated to $\alpha$ gives a branched cover $S\to \C/\Lambda'$, whose degree is at least $n_k+1$ and equals $\vol \chi / \mathrm{Area}(\C/\Lambda')$.

Conversely, given $\chi$ that satisfies those two conditions, then by \cref{resultat} there exists a translation surface with conical angle $2\pi(n_1+1),\ldots,2\pi(n_k+1)$ and holonomy $\chi$. Its developing map gives a branched cover $S\to \C/\Lambda$ with induced map on homology $\chi$. The degree of this covering map is $\vol\chi/\mathrm{Area}(\C/\Lambda) = d$, and the orders of its branching points are $n_1, \ldots, n_k$. 
\end{proof}
Note that these conditions on the existence of branched covers are not new, see for example \cite{Pervova}. Therefore \cref{corcor} is actually a characterisation of the induced maps on homology of the branched covers of elliptic curves.

Let us now explain the strategy of the proof.
We denote by $\Mod(S)$ the mapping class group of $S$.
Our proof relies on the study of the $\Mod(S)$ action on $\Hom(\Gamma, \C)$, that was carried out by Kapovich in \cite{Kapovich}, using Ratner's theory. Starting from an arbitrary $\chi\in \Hom(\Gamma, \C)$, Kapovich finds $\chi'\in \Mod(S)\cdot\chi$ with special properties that allows him to put a translation surface structure on $S$ with holonomy $\chi'$. 
This strategy is the same as Haupt's original one, except for the use of Ratner's theory.
We will adapt the discussion of the classification of $\Mod(S)$ orbits by Kapovich for  our purpose in \cref{Modaction}. Then we will use this classification to construct translation surface structures with given holonomy and given list of conical angles in \cref{construction}. In \cite{Kapovich}, the classification of the orbits holds only for $g\geqslant 3$, hence we will study the genus $2$ case without Ratner's theory in the spirit of Haupt's proof in \cref{genre2}.

\subsection*{Acknowledgements}

I would like to thank my advisor Maxime Wolff for his help, and Bertrand Deroin for his advice. While I was writing the present article, an independent proof of \cref{resultat} was published online by Bainbridge, Johnson, Judge and Park \cite{Judge}. I wish to thank them for their encouragement to post the present article. I also thank Julien March\'e, L\'eo B\'enard and Thibaut Mazuir for helpful conversations.

\section{Action of the mapping class group}\label{Modaction}

We will often identify $\Hom(\Gamma, \C)$ and $\C^{2g}$ with the map $\chi\mapsto (\chi(a_1),\ldots, \chi(b_g))$.
The group $\mathrm{GL}^+_2(\R)$ acts diagonally on $ (\R^2)^{2g} \simeq \C^{2g} \simeq \Hom(\Gamma, \C)$. A theorem of Dehn, Nielsen, and Baer, see \cite[Chapter 8]{FarbMargalit}, states that the natural map from $\Mod(S)$ to the positive outer automorphisms $\mathrm{Out}^+(\Gamma)$ of $\Gamma$ is an isomorphism. This latter group acts naturally on $\Hom(\Gamma, \C)$ by precomposition and thus yields a $\Mod(S)$ action on $\Hom(\Gamma, \C)$ that commutes with the $\mathrm{GL}_2^+(\R)$ action. Recall that $\mathrm{Out}^+(\Gamma)$ acts as $\mathrm{Sp}_{2g}(\Z)$ on $\Gamma^{\mathrm{ab}} \simeq H_1(S, \Z)\simeq \Z^{2g}$ and hence on $\C^{2g}$, see \cite[Chapter 6]{FarbMargalit}.
Let us state results obtained by Kapovich in \cite{Kapovich} using Ratner's theory.

\begin{theo}[Kapovich]\label{Kapovich} Let $\chi\in \Hom(\Gamma, \C)$ be such that $\vol \chi >0$.  After replacing $\chi$ by $A\cdot \chi$ with $A\in \mathrm{GL}_2^+(\R)$ if necessary, then
\begin{enumerate}
\item If $\chi(\Gamma)$ is not a lattice, there exists $\gamma\in \Mod(S)$ such that $\chi' =\gamma\cdot\chi$ satisfies $\Im(\overline{\chi'(a_i)} \chi'(b_i)) > 0$ for all $1 \leqslant i\leqslant g$.
\item If $\Lambda = \chi(\Gamma)$ is a lattice and $\vol \chi \geqslant 2 \mathrm{Area}(\C/\Lambda)$, then there exists $\gamma\in \Mod(S)$ such that $\chi' =\gamma\cdot \chi$ satisfies:
\begin{enumerate}
\item $\chi'(\Gamma)=\mathbb Z^2$.
\item $\chi'(a_1)$ is a positive integer and $\chi'(b_1) = i$.
\item $0 < \chi'(a_i) < \chi'(a_1)$ is an integer and $\chi'(b_i) = 0$ for all $i\geqslant 2$.
\end{enumerate}
\end{enumerate}
\end{theo}

We will need the following improvement of the lattice case.
\begin{prop}\label{ameliorerlattice}
In the second case of \cref{Kapovich}, we can moreover assume that $\chi'(a_i) = m_i$ for all $2 \leqslant i\leqslant g$, where $m_i\in \{1,2\}$, and $m_2=1$.
\end{prop}

We will give an alternative proof of the second part \cref{Kapovich} as well, based on the following known result.
\begin{lemma}\label{transitive}
The group $\mathrm{Sp}_{2g}(\Z)$ acts transitively on the vectors primitive vectors of $\Z^{2g}$.
\end{lemma}
We refer to \cite[Proposition 6.2]{FarbMargalit} for a proof of \cref{transitive}. We infer the classification of the orbits  of the $\mathrm{Sp}_{2g}(\Z)$ action on $\Z^{2g}$.

\begin{cor}
Let  $u,v\in \Z^{2g}\setminus \{0\}$. There exists $A\in \mathrm{Sp}_{2g}(\Z)$ such that $Au = v$ if and only if $\gcd(u) = \gcd(v)$.\qed
\end{cor}

We now prove \cref{ameliorerlattice}. Let $\chi\in \Hom(\Gamma, \C)$ be such that $\Lambda =\chi(\Gamma)$ is a lattice and such that $\vol\chi\geqslant 2\mathrm{Area}(\C/\Lambda)$.
\begin{proof}
We may assume that $\chi(\Gamma) = \Z\oplus i\Z$, changing $\chi$ by $A\cdot\chi$ with $A\in \mathrm{GL}_2^+(\R)$ if necessary. Let $u = (\Re(\chi(a_1)), \ldots, \Re(\chi(b_g))$ and $v = (\Im(\chi(a_1)), \ldots, \Im(\chi(b_g))$. Since $\gcd(u) = 1$, there exists $A\in \mathrm{Sp}_{2g}(\Z)$ such that $Av = (0, 1, 0, \ldots, 0)$. Moreover, this does not change if we multiply $A$ by $B = \mathrm{Id}_2\oplus C$, with $C\in \mathrm{Sp}_{2g-2}(\Z)$. We can choose $C$ such that $BAu = (p, q, 0, l, 0, 0, \ldots, 0, 0)$ with $l,p,q\in \Z$. Therefore there exists $\gamma\in \Mod(S)$ such that after replacing $\chi$ with $\gamma\cdot \chi$, we have $\chi(a_1) = p$, $\chi(b_1) = q+i$, $\chi(a_i) =\chi(b_i) = 0$ for $i\geqslant 2$ unless $i=2$, where $\chi(b_i) = l$. Since $i\in \chi(\Gamma)$, there exists $\lambda,\mu\in \Z$ such that $q + \lambda p + \mu l = 0$. We apply the matrix $$ \begin{pmatrix}
1 & 0 & -\mu & 0\\
0 & 1 & 0 & 0\\
0 & 0 & 1 & 0\\
0 & \mu & 0 & 1
\end{pmatrix}\oplus \mathrm{Id}_{2g-4}$$
that replaces $(a_1, b_1, a_2, b_2)$ with $(a_1, b_1+\mu b_2, a_2- \mu a_1, b_2)$. After this, replace $(a_1, b_1)$ by $(a_1, b_1 + \lambda a_1)$ with the matrix $\begin{pmatrix}
1 & \lambda\\
0 & 1
\end{pmatrix}\oplus \mathrm{Id}_{2g-2}$. We now have $\chi = (p, i, -\mu p, l, 0, \ldots, 0)$,. Changing $(a_2, b_2)$ to $(a_2, b_2 \pm a_2)$ or to $(a_2\pm b_2, b_2)$, one can make $\max(|\chi(a_2)|, |\chi(b_2)|)$ decrease until one of $\chi(a_2)$ and $\chi(b_2)$ is $0$, since $\chi(\Gamma)$ is discrete. Therefore we may assume that $\chi(a_2) = 0$, and that $\chi(b_2) = l$ for some $l\in \Z$. We replace $(a_1, b_1, a_2, b_2)$ with $(a_1, b_1 - b_2, a_2 + a_1, b_2)$ so that $\chi(a_2) = p$. Since $\chi(\Gamma) = \Z\oplus i\Z$, we have $\gcd(p, l) = 1$. We perform the euclidean algorithm as above in the handle $(a_2, b_2):$ replace $(a_2, b_2)$ with $(a_2, b_2 \pm a_2)$ or $(a_2\pm b_2, b_2)$ so that $\max(|\chi(a_2)|, |\chi(b_2)|)$ decreases until $\chi(a_2) = 0$ or $\chi(b_2)=0$. We can assume that $\chi(a_2) = 0$, and $\chi(b_2) = 1$, since $\gcd(\chi(a_2), \chi(b_2))$ is preserved at each step. Indeed, one can change the signs of $\chi(a_2)$ and $\chi(b_2)$ if necessary by applying the matrix $\mathrm{Id}_2 \oplus -\mathrm{Id}_2 \oplus \mathrm{Id}_{2g-4}$. Now replace $(a_1, b_1, a_2, b_2)$ by $(a_1, b_1 + kb_2, a_2 - ka_1, b_2)$, where $k$ is such that we get $\chi(b_1) = i$. We finally replace $(a_2, b_2)$ with $(b_2, -a_2 - kp b_2)$, so that $\chi(a_2) = 1$ and $\chi(b_2) = 0$. Therefore we have $\chi = (p, i, 1, 0, 0,\ldots, 0)$.
For each $3\leqslant i\leqslant g$, we perform the symplectic transformations $(a_2, b_2, a_i, b_i)\rightarrow(a_2, b_2 - m_ib_i, a_i + m_i a_2, b_i)$, so that $\chi(a_i) = m_i$ and $\chi(b_i) = 0$.
\end{proof}

From now on until \cref{genre2}, we assume that $g\geqslant 3$.
We also adapt the first case of \cref{Kapovich}.
\begin{prop}\label{ameliorer}
In the first case of the above theorem, we also can assume the following. 
\begin{enumerate}
\item The parallelogram formed by the vectors $\chi'(a_1)$ and $\chi'(b_1)$ contains in its interior a translate of a rectangle $R = I\times J\subset \R^2\simeq \C$ that contains all the points $M\chi'(a_i)$ for $i\geqslant 2$, where $M > 0$ is fixed.

\item $\arg(\chi(a_i)) \neq \arg(\chi(a_j))$ for every $2\leqslant i < j\leqslant g$.
\end{enumerate}
\end{prop}

\begin{proof}
Reviewing the arguments of Kapovich in \cite{Kapovich}, Calsamiglia, Deroin and Francaviglia showed that when we are not in the second case of \cref{Kapovich}, applying $A\in \mathrm{GL}_2^+(\R)$ if necessary, the orbit closure $\overline{\Mod(S)\cdot\chi}$ contains all $\chi'\in \C^{2g}$ such that $\vol{\chi'} = \vol \chi$ and such that $\Im(\chi')$ is in $\mathbb Z^{2g}$ and is primitive (see \cite[Proposition 8.1]{deroin}). Take $\chi'\in \Hom(\Gamma, \C)$ such that $\Im(\chi')\in \mathbb Z^{2g}$ is primitive, such that $\Im(\overline{\chi'(a_i)}\chi(b_i)) > 0$ for all $i$ and such that $\chi'$ satisfies the two conditions of \cref{ameliorer}. Rescale $\Re(\chi')$ by a positive constant so that $\vol{\chi'} = \vol\chi$. Then $\chi'\in \overline{\Mod(S)\cdot \chi}$, and $\chi'$ still satisfies the required properties. Since these conditions are open, we may assume that $\chi'\in\Mod(S)\cdot\chi$.
\end{proof}

\section{Construction of translation surfaces}\label{construction}

Let us fix a partition $1\leqslant n_1\leqslant n_2 \leqslant \ldots \leqslant n_k$ of $2g-2$ and $\chi \in \Hom(\Gamma, \C)$ such that $\vol \chi > 0$ and $\vol \chi \geqslant (n_k+1) \textrm{Area}(\C / \Lambda)$ if $\chi(\Gamma) = \Lambda$ is a lattice in $\C$. We observe that both the actions of $\Mod(S)$ and $\mathrm{GL}_2^+(\R)$ on $\Hom(\Gamma, \C)$ preserve the property of being in the image by $\Per$ of a given stratum.

\begin{lemma}
Let $\gamma\in \Mod(S)$ and $A\in \mathrm{GL}_2^+(\R)$. We have:
\begin{enumerate}
\item $\chi\in \mathrm{Per}(\mathcal H(n_1, \ldots, n_k)) \iff A\cdot \chi \in\mathrm{Per}(\mathcal H(n_1, \ldots, n_k))$
\item $\chi\in \mathrm{Per}(\mathcal H(n_1, \ldots, n_k)) \iff \gamma\cdot \chi \in \mathrm{Per}(\mathcal H(n_1, \ldots, n_k))$.
\end{enumerate}
\end{lemma}

\begin{proof}
Suppose $\chi$ is the holonomy of a translation surface structure. If $z$ is some local coordinate then define a new chart by $Az$. This gives a new translation surface structure with holonomy $A\cdot\chi$ and the same combinatoric of conical points. If $\gamma$ is represented by a homeomorphism $f$, then the pullback of the translation surface structure by $f$ gives a new translation structure with holonomy $\gamma\cdot \chi$.
\end{proof}

Note that the action of $\mathrm{GL}_2^+(\R)$ changes the volume, but does not change its sign, nor the ratio $\vol \chi/\textrm{Area}(\C/\Lambda)$ in the case where $\chi(\Gamma) = \Lambda$ is a lattice.

Therefore it suffices to consider the $\chi$ which have the form of \cref{ameliorer} or of \cref{ameliorerlattice} to prove \cref{resultat}. Our goal is to construct a translation surface with holonomy $\chi$ and with conical points of angles $2\pi(n_1+1), \ldots, 2\pi(n_k+1)$.
\subsection{The generic case}

In this section we suppose that $\chi$ has the form of the first case of \cref{ameliorer}.

Consider the parallelogram $P_i\subset \C$ \textit{associated} to $\chi(a_i)$ and $\chi(b_i)$ \textit{i.e.} whose vertices are $0$, $\chi(a_i)$, $\chi(b_i)$, $\chi(a_i)+\chi(b_i)$. This gives complex coordinates on the torus $T_1$ obtained by gluing the opposite sides of $P_1$ by translations. 

Before going on to the general construction, let us explain how to glue a torus to $T_1$ by adding a single conical point of angle $8\pi$. Make a slit in $P_1$, on a segment of the form $p + \chi(a_2)$. Identify the boundary components of these slits with the edges of $P_2$ corresponding to $\chi(a_2)$, see \cref{gluing}. We also identify the other edges of $P_2$ by the translation $z\mapsto z + \chi(a_2)$.
\begin{figure}[h]
    \centering    
    \def\svgwidth{\columnwidth}
	\scalebox{1}{
\begingroup%
  \makeatletter%
  \providecommand\color[2][]{%
    \errmessage{(Inkscape) Color is used for the text in Inkscape, but the package 'color.sty' is not loaded}%
    \renewcommand\color[2][]{}%
  }%
  \providecommand\transparent[1]{%
    \errmessage{(Inkscape) Transparency is used (non-zero) for the text in Inkscape, but the package 'transparent.sty' is not loaded}%
    \renewcommand\transparent[1]{}%
  }%
  \providecommand\rotatebox[2]{#2}%
  \newcommand*\fsize{\dimexpr\f@size pt\relax}%
  \newcommand*\lineheight[1]{\fontsize{\fsize}{#1\fsize}\selectfont}%
  \ifx\svgwidth\undefined%
    \setlength{\unitlength}{125.65555114bp}%
    \ifx\svgscale\undefined%
      \relax%
    \else%
      \setlength{\unitlength}{\unitlength * \real{\svgscale}}%
    \fi%
  \else%
    \setlength{\unitlength}{\svgwidth}%
  \fi%
  \global\let\svgwidth\undefined%
  \global\let\svgscale\undefined%
  \makeatother%
  \begin{picture}(1,0.33188779)%
    \lineheight{1}%
    \setlength\tabcolsep{0pt}%
    \put(0,0){\includegraphics[width=\unitlength,page=1]{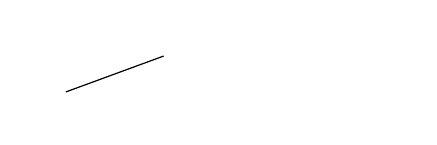}}%
    \put(0.1616516,0.09844962){\color[rgb]{0,0,0}\rotatebox{20.154313}{\makebox(0,0)[lt]{\lineheight{1.25}\smash{\begin{tabular}[t]{l}$+~+~+~+~+~+$\end{tabular}}}}}%
    \put(0.14946559,0.12874069){\color[rgb]{0,0,0}\rotatebox{20.154313}{\makebox(0,0)[lt]{\lineheight{1.25}\smash{\begin{tabular}[t]{l}$-~-~-~-~-~-$\end{tabular}}}}}%
    \put(0,0){\includegraphics[width=\unitlength,page=2]{imbis.pdf}}%
    \put(0.6096903,0.22167659){\color[rgb]{0,0,0}\rotatebox{20.154313}{\makebox(0,0)[lt]{\lineheight{1.25}\smash{\begin{tabular}[t]{l}$-~-~-~-~-~-$\end{tabular}}}}}%
    \put(0,0){\includegraphics[width=\unitlength,page=3]{imbis.pdf}}%
    \put(0.68739453,0.09828696){\color[rgb]{0,0,0}\rotatebox{20.154313}{\makebox(0,0)[lt]{\lineheight{1.25}\smash{\begin{tabular}[t]{l}$+~+~+~+~+~+$\end{tabular}}}}}%
    \put(0,0){\includegraphics[width=\unitlength,page=4]{imbis.pdf}}%
  \end{picture}%
\endgroup%
}
	\caption{Gluing a handle.}
    \label{gluing}
\end{figure}

The two end points of the slit on $P_1$ are identified, and yield a conical point of angle $8\pi$ on the resulting genus $2$ surface. Observe that its holonomy is given by the vector $(\chi(a_1), \chi(b_1), \chi(a_2), \chi(b_2))$, see \cref{topres}.

\begin{figure}[h]
    \centering    
    \def\svgwidth{\columnwidth}
	\scalebox{0.7}{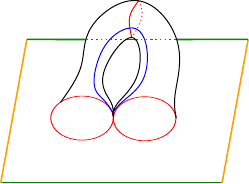}
    \caption{Topological result.}
    \label{topres}
\end{figure}

Let us return to the general case. We want to have conical singularities corresponding to $n_1, \ldots, n_k$. For each $1\leqslant i\leqslant k$, we glue $l_i$ tori as above on a point $p_i\in P$, where $2l_i = n_i$ or $2l_i = n_i-1$, see the black starfish formed by the slits meeting at $p_i$ in \cref{anse}. We make sure that these starfish do not intersect \textit{i.e.} that the slits do not overlap outside of the points $p_i$.

\begin{figure}[h]
    \centering    
    \def\svgwidth{\columnwidth}
	\scalebox{1}{
\begingroup%
  \makeatletter%
  \providecommand\color[2][]{%
    \errmessage{(Inkscape) Color is used for the text in Inkscape, but the package 'color.sty' is not loaded}%
    \renewcommand\color[2][]{}%
  }%
  \providecommand\transparent[1]{%
    \errmessage{(Inkscape) Transparency is used (non-zero) for the text in Inkscape, but the package 'transparent.sty' is not loaded}%
    \renewcommand\transparent[1]{}%
  }%
  \providecommand\rotatebox[2]{#2}%
  \newcommand*\fsize{\dimexpr\f@size pt\relax}%
  \newcommand*\lineheight[1]{\fontsize{\fsize}{#1\fsize}\selectfont}%
  \ifx\svgwidth\undefined%
    \setlength{\unitlength}{112.85757569bp}%
    \ifx\svgscale\undefined%
      \relax%
    \else%
      \setlength{\unitlength}{\unitlength * \real{\svgscale}}%
    \fi%
  \else%
    \setlength{\unitlength}{\svgwidth}%
  \fi%
  \global\let\svgwidth\undefined%
  \global\let\svgscale\undefined%
  \makeatother%
  \begin{picture}(1,0.41210103)%
    \lineheight{1}%
    \setlength\tabcolsep{0pt}%
    \put(0,0){\includegraphics[width=\unitlength,page=1]{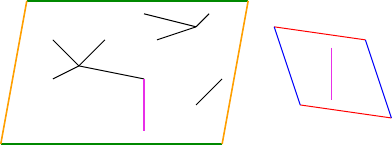}}%
    \put(0.34478626,0.19071079){\color[rgb]{0,0,0}\rotatebox{-89.299634}{\makebox(0,0)[lt]{\lineheight{1.25}\smash{\begin{tabular}[t]{l}$+~+~+$\end{tabular}}}}}%
    \put(0.37557882,0.20263167){\color[rgb]{0,0,0}\rotatebox{-89.299634}{\makebox(0,0)[lt]{\lineheight{1.25}\smash{\begin{tabular}[t]{l}$-~-~-$\end{tabular}}}}}%
    \put(0.81912251,0.28040686){\color[rgb]{0,0,0}\rotatebox{-89.299634}{\makebox(0,0)[lt]{\lineheight{1.25}\smash{\begin{tabular}[t]{l}$-~-~-$\end{tabular}}}}}%
    \put(0.8566281,0.27772504){\color[rgb]{0,0,0}\rotatebox{-89.299634}{\makebox(0,0)[lt]{\lineheight{1.25}\smash{\begin{tabular}[t]{l}$+~+~+$\end{tabular}}}}}%
    \put(0,0){\includegraphics[width=\unitlength,page=2]{gengen.pdf}}%
  \end{picture}%
\endgroup%
}
    \caption{Construction of $\alpha\in \mathcal{H}(5, 6, 9)$.}
    \label{anse}
\end{figure}We now explain how to introduce odd conical points by gluing tori. We can glue the torus obtained by identifying the opposite sides of $P_i$ by translations to $T_1$, by making a slit in both of them and gluing the boundary components as indicated in purple in \cref{anse}. This adds a $2\pi$ angle at each end of the slit, and allows us to have the desired conical angles. We may assume that the starfish are sufficently close to make these slits, translating them if necessary.
By construction, the resulting translation surface has holonomy $\chi$, and the corresponding abelian differential is in $\mathcal H(n_1, \ldots, n_k)$.
Note that we have room in $P_1$ to make these gluings if we take $M$ sufficently large in \cref{ameliorer}.

\subsection{The lattice case}

We now suppose that $\chi$ has the form of \cref{ameliorerlattice}. Note that $\vol \chi = \chi(a_1)$.

\subsubsection{Multiple singularities}

We assume here that $k\geqslant 2$. Let us first explain how to build a translation surface with corresponding abelian differential in $\mathcal H(g-1, g-1)$. Consider the parallelogram $P$ in $\C$ associated to $\chi(a_1)$ and $\chi(b_1)$.
Consider a segment on $P$ such that its translates by $z\mapsto z + k$, with $1\leqslant k\leqslant g-1$ are also contained in $P$. Cutting $P$ along those segments yields $2g$ boundaries $\delta_1^+$, $\delta_1^-, \ldots, \delta_{g}^+, \delta_{g}^-$. Identify $\delta_i^+$ with $\delta_{i+1}^-$ for all $1 \leqslant i\leqslant g$, in cyclic notation, as in \cref{hgg}.
\begin{figure}[h]
    \centering    
    \def\svgwidth{\columnwidth}
	\scalebox{0.7}{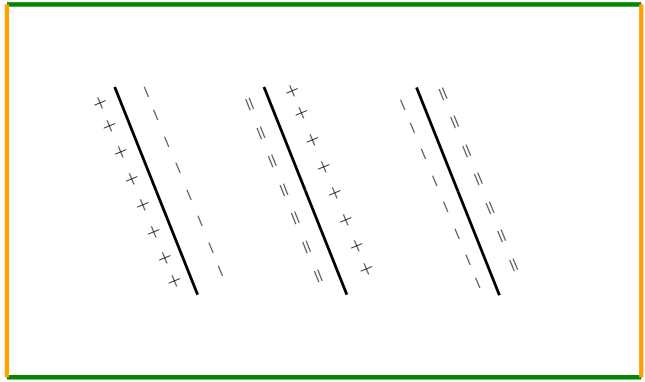}
    \caption{Construction of $\alpha\in \mathcal H(2, 2)$.}
    \label{hgg}
\end{figure}

Identifying the edges of $P$ by translations, we get a flat surface with two conical points of angle $2g\pi$. The holonomy of this surface is $(\vol \chi, i, 1, 0,\ldots, 1, 0)$, see \cref{toptopres},  and is in $\Mod(S)\cdot\chi$ by \cref{ameliorerlattice}.
Note that is was possible to put these segments in $P$ because of the assumption that $\vol \chi > g-1$.
\begin{figure}[h]
    \centering    
    \def\svgwidth{\columnwidth}
	\scalebox{0.7}{
\begingroup%
  \makeatletter%
  \providecommand\color[2][]{%
    \errmessage{(Inkscape) Color is used for the text in Inkscape, but the package 'color.sty' is not loaded}%
    \renewcommand\color[2][]{}%
  }%
  \providecommand\transparent[1]{%
    \errmessage{(Inkscape) Transparency is used (non-zero) for the text in Inkscape, but the package 'transparent.sty' is not loaded}%
    \renewcommand\transparent[1]{}%
  }%
  \providecommand\rotatebox[2]{#2}%
  \newcommand*\fsize{\dimexpr\f@size pt\relax}%
  \newcommand*\lineheight[1]{\fontsize{\fsize}{#1\fsize}\selectfont}%
  \ifx\svgwidth\undefined%
    \setlength{\unitlength}{64.18243561bp}%
    \ifx\svgscale\undefined%
      \relax%
    \else%
      \setlength{\unitlength}{\unitlength * \real{\svgscale}}%
    \fi%
  \else%
    \setlength{\unitlength}{\svgwidth}%
  \fi%
  \global\let\svgwidth\undefined%
  \global\let\svgscale\undefined%
  \makeatother%
  \begin{picture}(1,0.72634974)%
    \lineheight{1}%
    \setlength\tabcolsep{0pt}%
    \put(0,0){\includegraphics[width=\unitlength,page=1]{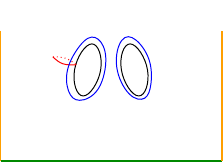}}%
    \put(0.24438101,0.43234721){\color[rgb]{0,0,0}\makebox(0,0)[lt]{\lineheight{1.25}\smash{\begin{tabular}[t]{l}$b_2$\end{tabular}}}}%
    \put(0.34245522,0.26019572){\color[rgb]{0,0,0}\makebox(0,0)[lt]{\lineheight{1.25}\smash{\begin{tabular}[t]{l}$a_2$\end{tabular}}}}%
    \put(0.64815472,0.27167242){\color[rgb]{0,0,0}\makebox(0,0)[lt]{\lineheight{1.25}\smash{\begin{tabular}[t]{l}$a_3$\end{tabular}}}}%
    \put(0.56886064,0.60710712){\color[rgb]{0,0,0}\makebox(0,0)[lt]{\lineheight{1.25}\smash{\begin{tabular}[t]{l}$b_3$\end{tabular}}}}%
    \put(0,0){\includegraphics[width=\unitlength,page=2]{toptopres.pdf}}%
  \end{picture}%
\endgroup%
}
    \caption{Topological result.}
    \label{toptopres}
\end{figure}

Let us return to the general case. We start as before, by adding $n_1$ handles to the parallelogram $P$ associated to $\chi(a_1)$, $\chi(b_1)$ and thus constructing a flat surface of genus $n_1+1$, with two conical points of angle $2\pi(n_1+1)$. We then make $n_2 - n_1$ other parallel slits below, such that one of them touches one end of the rightmost preceding slit, that we glue in a cyclic way (see the top of \cref{general}). In doing so, we obtain a flat surface with conical angles $2\pi(n_1+1)$ and $2\pi(n_2 + 1)$ and $2\pi(n_2-n_1+1)$. Iterating this process, we get a flat surface with angles $2\pi(n_1+1), \ldots, 2\pi(n_{k-1}+1)$ and $2\pi(l+ 1)$, with $l\leqslant n_k$. If needed, make a slit of length one joining the rightmost two ends of the last slits as in the bottom of \cref{general}, and some other slits that are translations by multiples of $2$ of this one, that we glue in a cyclic way to adjust the last conical angle.

\begin{figure}[h]
    \centering    
    \def\svgwidth{\columnwidth}
	\scalebox{1}{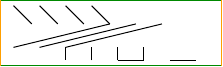}
    \caption{Construction of $\alpha\in \mathcal H(3, 5, 5, 5)$.}
    \label{general}
\end{figure}

Each horizontal line has length at most $n_k$, hence we have room to make the slits, since $\vol \chi = \chi(a_1) \geqslant\ n_k+1$. We check that the holonomy of the resulting flat surface is given by $(\vol\chi, i, 1, 0, \ldots, 1, 0, 2, 0, \ldots, 2, 0)$ and thus is in $\Mod(S)\cdot \chi$ by \cref{ameliorerlattice}.

\subsubsection{The minimal stratum}

We now suppose that $k=1$. Thanks to \cref{ameliorerlattice} we can assume that $\chi(a_2) = 1$ and that $\chi(a_i) = 2$ for $i \geqslant 3$. We start as before by considering the parallelogram $P$ associated to $\chi(a_1)$ and $\chi(b_1)$ in $\C$, that gives complex charts on the associated torus.
We make a horizontal slit of length one in $P$. We make other slits along segments which are translates of the first one by $n = 1, 3, 5, 7\ldots, 2g-3$, see \cref{min}. We identify the boundaries in a cyclic way as before.

\begin{figure}[h]
    \centering    
    \def\svgwidth{\columnwidth}
	\scalebox{1}{
\begingroup%
  \makeatletter%
  \providecommand\color[2][]{%
    \errmessage{(Inkscape) Color is used for the text in Inkscape, but the package 'color.sty' is not loaded}%
    \renewcommand\color[2][]{}%
  }%
  \providecommand\transparent[1]{%
    \errmessage{(Inkscape) Transparency is used (non-zero) for the text in Inkscape, but the package 'transparent.sty' is not loaded}%
    \renewcommand\transparent[1]{}%
  }%
  \providecommand\rotatebox[2]{#2}%
  \newcommand*\fsize{\dimexpr\f@size pt\relax}%
  \newcommand*\lineheight[1]{\fontsize{\fsize}{#1\fsize}\selectfont}%
  \ifx\svgwidth\undefined%
    \setlength{\unitlength}{65.71245471bp}%
    \ifx\svgscale\undefined%
      \relax%
    \else%
      \setlength{\unitlength}{\unitlength * \real{\svgscale}}%
    \fi%
  \else%
    \setlength{\unitlength}{\svgwidth}%
  \fi%
  \global\let\svgwidth\undefined%
  \global\let\svgscale\undefined%
  \makeatother%
  \begin{picture}(1,0.28964778)%
    \lineheight{1}%
    \setlength\tabcolsep{0pt}%
    \put(0,0){\includegraphics[width=\unitlength,page=1]{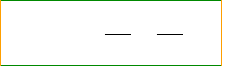}}%
    \put(0.15064641,0.15113523){\color[rgb]{0,0,0}\makebox(0,0)[lt]{\lineheight{1.25}\smash{\begin{tabular}[t]{l}$+~+~~-~-$\end{tabular}}}}%
    \put(0.13904119,0.10934204){\color[rgb]{0,0,0}\makebox(0,0)[lt]{\lineheight{1.25}\smash{\begin{tabular}[t]{l}$-~-~~=~=$\end{tabular}}}}%
    \put(0.47851984,0.15139227){\color[rgb]{0,0,0}\makebox(0,0)[lt]{\lineheight{1.25}\smash{\begin{tabular}[t]{l}$=~=$\end{tabular}}}}%
    \put(0.48737319,0.11115593){\color[rgb]{0,0,0}\makebox(0,0)[lt]{\lineheight{1.25}\smash{\begin{tabular}[t]{l}$*~~~*$\end{tabular}}}}%
    \put(0.71368511,0.15666807){\color[rgb]{0,0,0}\makebox(0,0)[lt]{\lineheight{1.25}\smash{\begin{tabular}[t]{l}$*~~~*$\end{tabular}}}}%
    \put(0.70996641,0.10963857){\color[rgb]{0,0,0}\makebox(0,0)[lt]{\lineheight{1.25}\smash{\begin{tabular}[t]{l}$+~~+$\end{tabular}}}}%
    \put(0,0){\includegraphics[width=\unitlength,page=2]{min.pdf}}%
  \end{picture}%
\endgroup%
}
    \caption{Construction of $\alpha\in \mathcal H(6)$.}
    \label{min}
\end{figure}

We obtain a flat surface, with a single conical point of angle $2\pi(2g-1)$.
The total length of this construction is $2g-2$. Therefore, we have room to make it in $P$, since $\vol \chi \geqslant 2g-1$. We thus get a genus $g$ translation surface with holonomy $\chi$ and a single conical point.

\section{The genus 2 case}\label{genre2}
The aim of this section is to prove \cref{resultat} when $g=2$.
Let $\chi = (a_1, b_1, a_2, b_2)\in \C^4\simeq \Hom(\Gamma, \C)$ be such that $\vol \chi > 0$. We assume that $\chi(\Gamma)$ is not a lattice in $\C$, since otherwise the above constructions, together with \cref{ameliorerlattice}, show that \cref{resultat} holds.

Let us observe that by definition, $\vol\chi = \det(a_1, b_1) + \det(a_2, b_2)$. We may assume that $\det(a_1, b_1) \geqslant \det(a_2, b_2)$, applying $[f]\in \Mod(S)$ that interchanges the two handles if necessary. Since $\vol \chi > 0$, we have $\det(a_1, b_1) > 0$. Let us show that it suffices to have $\det(a_1, b_1)$ and $\det(a_2, b_2)$ positive, with the ratio $\det(a_2, b_2)/\det(a_1, b_1)$ small enough, to prove \cref{resultat}.

\begin{lemma}\label{SL2}
If $v_1$ and $v_2$ generate a lattice in $\C$, then there exists $A\in \mathrm{SL}_2(\mathbb Z)$ such that $(v_1', v'_2) = A(v_1, v_2)$ satisfies $\|v_1'\|\leqslant C\sqrt{|\det(v_1, v_2)|}$ where $C = \sqrt{\frac{2}{\sqrt{3}}}$.
\end{lemma}

\begin{proof} We may assume that $\det(v_1, v_2) < 0$, replacing $v_1$ with $-v_1$ if necessary.
It suffices to show that after applying $A\in \mathrm{SL}_2(\mathbb Z)$, we can assume that the angle between $v_1$ and $v_2$ is $\theta\in [\frac{\pi}{3}, \frac{2\pi}{3}]$, since then $$\frac{\sqrt 3}{2}(\min(\|v_1\|, \|v_2\|)^2\leqslant\|v_1\| \|v_2\| \sin \theta=|\det(v_1, v_2)|.$$
Applying $(v_1, v_2)\rightarrow (v_2, -v_1)$ if necessary, we have the desired inequality.

Since the action of $\C^*$ on the pairs $(v_1, v_2)$ preserves the angles between the two vectors, we are reduced to the study of the $\mathrm{SL}_2(\mathbb Z)$ action on $\mathbb H^2$ by renormalizing $(v_1, v_2)$ so that $v_2 = 1$. It is well known that $D = \{z\in \mathbb H^2 \mid |z|> 1,-1/2 \leqslant\Re(z)\leqslant 1/2\}$ is a fundamental domain for this action, see for example \cite[Chapter 7]{Serre}. Therefore, we may assume that the angle between $v_1$ and $v_2$ is in $[\frac \pi 3, \frac{2\pi}{3}]$.
\end{proof}

\begin{prop}\label{ratiogagnant}
If $0 < \det(a_2, b_2)< \epsilon \det(a_1, b_1)$, then $\chi\in \Per(\mathcal H(2))$ and $\chi\in\Per (\mathcal H(1,1))$, where $\epsilon = \frac{1}{C^2}$.
\end{prop}

\begin{proof}
Applying $A\in \mathrm{GL}_2^+(\R)$ does not change the ratio $\det(a_2, b_2)/\det(a_1, b_1)$, hence we may assume that $(a_1, b_1) = (1, i)$. We thus have $\det(a_2, b_2) < \epsilon$, and it follows from \cref{SL2} that we can assume that $\|a_2\| < C\sqrt{\epsilon} \leqslant 1$. Now we can glue the parallelograms associated to $a_1, b_1$ and $a_2, b_2$ as in \cref{gluing} since we have room to make the slit. Hence $\chi\in \Per(\mathcal H(2))$. We can also glue these parallelograms with a slit as shown in purple in \cref{anse}, hence $\chi\in \Per(\mathcal H(1, 1)).$
\end{proof}

It is now enough to prove the following proposition.
\begin{prop}\label{proposition}
There exists $A\in \mathrm{GL}_2^+(\R)$ and $\gamma\in \Mod(S)$ such that $\chi' = A\cdot(\gamma\cdot \chi)$ satisfies the condition of \cref{ratiogagnant}.
\end{prop}
We may assume that $\det(a_1,b_1) > 0$ and, applying $A\in \mathrm{GL}_2^+(\R)$, that $a_1 = 1$ and $b_1 = i$.
\begin{lemma}
If the projection of the group generated by $a_2$ and $b_2$ on the real line is dense, then \cref{proposition} holds. 
\end{lemma}
\begin{proof}
There exists $n,m\in \Z$ such that $x = \Re(na_2 + mb_2)$ satisfies
$$\frac{\vol \chi}{1+\epsilon}< x+1 <\vol\chi.$$
Let $k = \gcd(n,m)$ and $n' = n/k$, $m' = m/k$.
We may replace $a_2$ with $n'a_2 + m'b_2$ by applying $A\in \mathrm{Sp}_4(\Z)$ of the form $A = \mathrm{Id}_2 \oplus B$ where $B\in \SL\Z$. 
We now apply the following transformation of $\mathrm{Sp}_4(\mathbb Z)$ : $(a_1, b_1, a_2, b_2)\rightarrow (a_1 + ka_2, b_1, a_2, b_2 - kb_1)$. Thus $\det(a_1, b_1) = x+1$, and $$0 < \det(a_2, b_2) = \vol \chi - \det(a_1, b_1)< \epsilon\det(a_1, b_1).\qedhere$$
\end{proof}

We can now suppose that $a_2 = x + iy$ and $b_2 = x' + iy'$ are such that the group generated by $x$ and $x'$ is discrete. We also assume that the group generated by $y$ and $y'$ is dense, otherwise we could return to the previous case by applying the transformation $(1, i, a_2, b_2)\rightarrow (-i, 1, a_2, b_2)$.
\begin{lemma}\label{euclide}
We can assume that $x' = 0$.
\end{lemma}
\begin{proof}
Apply the transformation $(a_2, b_2)\rightarrow (a_2 \pm b_2, b_2)$ or $(a_2, b_2)\rightarrow (a_2, b_2\pm a_2)$ so that $\max(|\Re(a_2)|, |\Re(b_2)|)$ decreases. Since the group generated by $x$ and $x'$ is discrete, iterating this process will end up with $x$ or $x' = 0$. Note that we have already used this Euclidean algorithm's argument in \cref{ameliorerlattice}.
\end{proof}

\begin{lemma}\label{diviser}
If both $\det(a_1, b_1)\geqslant\det(a_2, b_2)$ are not $0$, there exists $\gamma\in \Mod(S)$ such that $\chi' =\gamma\cdot\chi = (a'_1, b'_1, a'_2, b'_2)$ satisfies $$0\leqslant\det(a'_1, b'_1) \leqslant \frac{\det(a_1, b_1)}2.$$
Moreover, if $\det(a_1, b_1) \neq \det(a_2, b_2)$, we may assume that $0 < \det(a_1', b_1')$.
\end{lemma}

\begin{proof}
Let us assume again that $(a_1, b_1) = (1, i)$.
We have $0 <|xy'| = |\det(a_2, b_2)| \leqslant \det(a_1, b_1) = 1$ since $\det(a_1, b_1) + \det(a_2, b_2) = \vol\chi$, thus $0<|y'|\leqslant1$ or $0 < |x| \leqslant 1$, and $|x| < 1$ or $|y'| < 1$ if $\det(a_1, b_1)\neq \det(a_2, b_2)$. In the former case, replace $(a_1, b_1, a_2, b_2)$ with $(a_1, b_1 + kb_2, a_2 - kb_1, b_2)$, where $k$ is such that $0\leqslant (1+ky')\leqslant \frac 1 2$. Such an integer exists: it is equivalent to $\frac{1}{2}\leqslant-ky'\leqslant 1$ which is possible with $k$ an appropriate power of $2$. Note that if $|y'| < 1$, we may assume that $0 < 1 + ky'$.
In the second case, replace $\chi$ with $(a_1 + ka_2, b_1, a_2, b_2 - kb_1)$ where $0\leqslant 1+kx \leqslant \frac 1 2$. In both cases the new $\det(a_1, b_1)$ is less than half the old one.
\end{proof}

It might happen that $\det(a_2, b_2) = 0$. Let us explain how to proceed in this case.

\begin{lemma}\label{zero}
If $\det(a_2, b_2) = 0$, then there exists $\gamma\in \Mod(S)$ such that replacing $\chi$ with $\gamma\cdot\chi$, we have $\det(a_1, b_1) > 0$ and $\det(a_2, b_2) > 0$ and $\det(a_1, b_1)\neq \det(a_2, b_2)$.
\end{lemma}
\begin{proof}
We can suppose that $\chi = (1, i, z, \lambda z)$ for some $z = x + iy\in \C$, $\lambda\in \R$. Note that $z\neq 0$, otherwise $\chi(\Gamma)$ would be a lattice. Moreover, $\lambda\in \mathbb Q$, because otherwise the group generated by $x, \lambda x$ or $y$, $\lambda y$ would be dense. Thus we can assume that $\lambda = 0$ by the argument of \cref{euclide}. We thus have $\chi = (1, i, z, 0)$. Either $x\notin \mathbb Q$ or $y\notin \mathbb Q$ since $\vol \chi > 0$ and $\chi(\Gamma)$ is not a lattice. If $x\notin \mathbb Q$, apply the transformation $(a_1, b_1 - kb_2, a_2 + ka_1, b_2)$ so that $0< x < 1$. We apply $(a_1 -a_2, b_1, a_2, b_2 + b_1)$, that replaces $\det(a_1, b_1)$ with $0<\det(a_1, b_1) - x < 1$. Moreover, $\det(a_1, b_1) - x = 1-x\neq \frac{1}{2}$ since $x\notin \mathbb Q$.
If $y\notin \mathbb Q$, we can proceed in the same way after applying $(a_1, b_1)\rightarrow (b_1, -a_1)$.
The result follows since $\det(a_1, b_1) + \det(a_2, b_2) = 1$.
\end{proof}

We are now able to prove \cref{proposition}. Let us normalize the volume: $\vol\chi = 1$.
\begin{proof}
Iterating \cref{diviser}, and using \cref{zero} when necessary, we can assume that $x = \det(a_1, b_1) > 0$ and $\det(a_2, b_2) > 0$, \textit{i.e.} that $\frac{1}{2} \leqslant x < 1$. We can moreover assume that $x \neq \frac{1}{2}$. Indeed if $x=\frac{1}{2}$, use one more time \cref{diviser}, and then \cref{zero} if necessary.

It suffices to show that we can have $\frac{2}{3}\leqslant x < 1$. Indeed, in this case, $\frac{\det(a_2, b_2)}{\det(a_1, b_1)}\leqslant \frac{1-x}{x} \leqslant \frac{1}{2} < \epsilon$, and we conclude with \cref{ratiogagnant}. If $x < \frac{2}{3}$, then applying \cref{diviser}, we replace $x$ with $x'$ that satisfies $\frac{2}{3} < 1 - \frac x 2 \leqslant x' < 1$.
\end{proof}

\bibliographystyle{plain}
\bibliography{bibpapier.bib}

\end{document}